\documentclass[11pt]{amsart}
\usepackage[top=1.5in, bottom=1.5in, left=1.5in, right=1.5in]{geometry}
\expandafter\let\csname ver@amsthm.sty\endcsname\relax

\usepackage{amsmath}
\usepackage{amssymb}
\usepackage{hyperref}
\usepackage{amsthm}
\usepackage{mathtools}
\usepackage[capitalize,noabbrev]{cleveref}
\usepackage{xcolor}
\usepackage{enumitem}

\allowdisplaybreaks

\numberwithin{equation}{section}

\newtheorem{thm}{Theorem}[section]
\newtheorem{lemma}[thm]{Lemma}
\newtheorem{cor}[thm]{Corollary}
\newtheorem{prop}[thm]{Proposition}

\newtheorem{Example}[thm]{Example}
\newenvironment{example}
  {\begin{Example}\rm}{\end{Example}}

\newtheorem{Remark}[thm]{Remark}
\newenvironment{remark}
  {\begin{Remark}\rm}{\end{Remark}}
  
\crefname{thm}{Theorem}{Theorems}
\crefname{lemma}{Lemma}{Lemmas}
\crefname{cor}{Corollary}{Corollaries}
\crefname{prop}{Proposition}{Propositions}

\crefname{example}{Example}{Examples}
\crefname{remark}{Remark}{Remarks}

\newcommand{\dfn}[1]{\textcolor{blue}{\emph{#1}}}

\title[Restricted Birkhoff polytopes and Ehrhart period collapse]{Restricted Birkhoff polytopes \\ and Ehrhart period collapse}
\author{Per Alexandersson}
\email{per.w.alexandersson@gmail.com}
\address{Department of Mathematics, Stockholm University, SE-106 91 Stockholm, Sweden}
\author{Sam Hopkins}
\email{samuelfhopkins@gmail.com}
\address{Department of Mathematics, Howard University, Washington, DC, USA}
\author{Gjergji Zaimi}
\email{gjergjiz@gmail.com}
\address{Palo Alto, California, USA}


\begin{document}

\begin{abstract}
We show that the polytopes obtained from the Birkhoff polytope by imposing additional inequalities restricting the ``longest increasing subsequence'' have Ehrhart quasi-polynomials which are honest polynomials, even though they are just rational polytopes in general. We do this by defining a continuous, piecewise-linear bijection to a certain Gelfand--Tsetlin polytope. This bijection is not an integral equivalence but it respects lattice points in the appropriate way to imply that the two polytopes have the same Ehrhart (quasi-)polynomials. In fact, the bijection is essentially the Robinson--Schensted--Knuth correspondence.
\end{abstract}

\maketitle

\section{Introduction and statement of the main result} \label{sec:intro}

Let $\mathcal{P}$ be a $d$-dimensional rational convex\footnote{From now on, all polytopes are assumed convex and we will drop this adjective.} polytope in $\mathbb{R}^n$. A celebrated theorem of Ehrhart~\cite{ehrhart1962sur} says that the function
\[ L(\mathcal{P};t) \coloneqq \#(t\mathcal{P}\cap\mathbb{Z}^n)\]
which counts the number of lattice points in the $t$th dilate of $\mathcal{P}$, for integers $t \geq 1$, is a quasi-polynomial in $t$ of degree $d$. Here by \dfn{quasi-polynomial} we mean a polynomial with periodic coefficients. That is, there is some $p$ for which we can write
\[ L(\mathcal{P};t) = c_d(t) t^d + c_{d-1}(t)t^{d-1} + \cdots + c_1(t)t + c_0(t)\]
where the coefficients $c_i(t)$ satisfy $c_i(t) = c_i(t+p)$ for all $t$. The \dfn{period} of the quasi-polynomial is the smallest such $p$ we can choose. The quasi-polynomial $L(\mathcal{P};t)$ is called the \dfn{Ehrhart quasi-polynomial} of $\mathcal{P}$ and it is known that the period of this quasi-polynomial divides the smallest integer $m\geq 1$ such that $m\mathcal{P}$ is integral. For example, if~$\mathcal{P}$ is already an integral polytope (so $m=1$), then $L(\mathcal{P};t)$ is an honest polynomial called the \dfn{Ehrhart polynomial} of $\mathcal{P}$. For more basics about polytopes and Ehrhart theory, we refer the reader to~\cite{beck2015computing}.

Notice how we said that the period of $L(\mathcal{P};t)$ \emph{divides} the smallest $m$ such that~$m\mathcal{P}$ is integral: generically it is equal to this $m$, but for special polytopes it can be smaller. When $\mathcal{P}$ has a smaller Ehrhart quasi-polynomial period than what one would predict by looking at the denominators of its vertices, we say that~$\mathcal{P}$ exhibits \dfn{Ehrhart period collapse}. In the most extreme case, the Ehrhart quasi-polynomial of~$\mathcal{P}$ may be an honest polynomial even when $\mathcal{P}$ is not integral.

Ehrhart period collapse has received a significant amount of research attention (see, e.g.,~\cite{beck2008maximal, haase2008quasiperiod, mcallister2017ehrhart, mcallister2005minimum}) but is still poorly understood in general. For instance, in~\cite{haase2008quasiperiod} it is conjectured that period collapse for a rational polytope can always be explained by cutting the polytope into simplices and rearranging them in the appropriate way to form another polytope with the ``correct'' denominators for its vertices. But this conjecture remains open.

In the absence of a general explanation for the phenomenon, there has been interest in studying particular instances of Ehrhart period collapse, especially for families of polytopes with important connections to combinatorics, geometry, or algebra~\cite{cristofarogardiner2019irrational, fernandes2022period, johnson2022piecewise}. Indeed, it is a curious fact that many prominent examples of polytopes exhibiting Ehrhart period collapse come from algebra. The \dfn{Gelfand--Tsetlin polytope}~$\mathcal{GT}_{\lambda,\mu}$, which depends on a choice of partitions $\lambda$ and $\mu$, consists of certain triangular arrays of numbers called Gelfand--Tsetlin patterns. The number of integral Gelfand--Tsetlin patterns (i.e., the number of lattice points in~$\mathcal{GT}_{\lambda,\mu}$) is the dimension of the $\mu$ weight space in the irreducible general linear group representation indexed by $\lambda$. This representation theory interpretation implies that $L(\mathcal{GT}_{\lambda,\mu};t)$ is a polynomial~\cite{billey2004vector, kirillov2001ubiquity}, even though~$\mathcal{GT}_{\lambda,\mu}$ is not in general integral~\cite{deloera2004vertices, king2004stretched}. In other words, $\mathcal{GT}_{\lambda,\mu}$ exhibits Ehrhart period collapse.

\medskip

In this article we describe a new family of polytopes exhibiting Ehrhart period collapse. These new examples are attractive because they are built from the famous Birkhoff polytope of doubly stochastic matrices.

Recall that the \dfn{Birkhoff polytope} $\mathcal{B}_n$ is the set of all $n\times n$ $\mathbb{R}$-matrices \[X=\begin{pmatrix} x_{1,1} & x_{1,2} & \cdots & x_{1,n} \\
x_{2,1} & x_{2,2} & \ddots & \vdots \\
\vdots & \ddots & \ddots & \vdots \\
x_{n,1} & \cdots & \cdots & x_{n,n}
\end{pmatrix}\] 
which satisfy the inequalities 
\[x_{i,j} \geq 0 \textrm{ \, for all $i,j$ \; (\emph{nonnegative entries})},\]
and equalities
\begin{align*}
x_{i,1} + x_{i,2} + \cdots + x_{i,n} &= 1 \textrm{ \, for all $i$ \; \textrm(\emph{row sums equal one})}, \\
x_{1,j} + x_{2,j} + \cdots + x_{n,j} &= 1 \textrm{ \, for all $j$ \; (\emph{column sums equal one})}.
\end{align*}
It is known that $\mathcal{B}_n$ is a $(n^2-2n+1)$-dimensional polytope inside of $\mathbb{R}^{n\times n}$. Moreover, the Birkhoff--von Neumann theorem asserts that the vertices of $\mathcal{B}_n$ are precisely the permutation matrices $X_w$ for permutations $w$ in the symmetric group $\mathfrak{S}_n$ (see, e.g., \cite[Example~0.12]{ziegler1995lectures}).

\begin{remark}
The Ehrhart polynomial of $\mathcal{B}_n$ was studied by Beck and Pixton~\cite{beck2003ehrhart} from a computational point of view, but not much is known beyond small values of~$n$. Indeed, even computing the \emph{volume} of~$\mathcal{B}_n$ is a very difficult problem~\cite{canfield2009asymptotic, chan1999volume, deloera2009generating, 
 pak2000four}. 
\end{remark}

Since $\mathcal{B}_n$ is an integral polytope, we will need to impose further inequalities if we want to find an example of Ehrhart period collapse. Therefore, for $k=1,2,\ldots,n$ we define the \dfn{restricted Birkhoff polytope} $\mathcal{B}^k_n$ to be the set of $X = (x_{i,j}) \in \mathcal{B}_n$ which satisfy the additional inequalities
\[ x_{i_1,j_2} + x_{i_2,j_2} + \cdots + x_{i_{2n-1},j_{2n-1}} \leq k \]
for all sequences $(1,1)=(i_1,j_1) \lessdot (i_2,j_2) \lessdot \cdots \lessdot (i_{2n-1},j_{2n-1})=(n,n)$, where we write $(i,j) \lessdot (i',j')$ to mean either $(i',j')=(i+1,j)$ or $(i',j')=(i,j+1)$. In other words, the sum of entries in any chain from the upper-left to lower-right corners must be at most $k$. These are $\binom{2n-2}{n-1}$ additional inequalities. 

Observe that $\mathcal{B}^n_n = \mathcal{B}_n$ is the full Birkhoff polytope, while $\mathcal{B}^1_n=\{X_{w_0}\}$ where $w_0 = n,n-1,\ldots,2,1 \in \mathfrak{S}_n$ is the reverse permutation. For other values of~$k$ we get something in-between the Birkhoff polytope~$\mathcal{B}_n$ and the single point~$X_{w_0}$.

A moment's thought will reveal that the vertices of $\mathcal{B}_n$ which belong to $\mathcal{B}^k_n$ are precisely $X_w$ for $w\in\mathfrak{S}_n$ whose longest increasing subsequence has length at most~$k$. In the language of permutation pattern avoidance, these are the permutations which are \dfn{$12\ldots (k+1)$-avoiding}. For example, in the case $k=2$ these are the $123$-avoiding permutations, and it is a classic result that there are Catalan number many of them. However, except in the cases $k=1$ and $k=n$, there are other vertices of $\mathcal{B}^k_n$ which are \emph{not} permutation matrices, that is, which are not integral.\footnote{Thus, compare to the recent work of Davis and Sagan~\cite{davis2018pattern} which studies the subpolytopes of the Birkhoff polytope determined by pattern-avoiding permutations.}

\begin{example} \label{ex:b32}
There are $6$ vertices of $\mathcal{B}_3^2$. They are the following:
\[\begin{pmatrix}1 & 0 & 0 \\ 0 & 0 & 1 \\ 0 & 1 & 0 \end{pmatrix}, \, \begin{pmatrix}0 & 1 & 0 \\ 1 & 0 & 0 \\ 0 & 0 & 1 \end{pmatrix}, \, \begin{pmatrix}0 & 1 & 0 \\ 0 & 0 & 1 \\ 1 & 0 & 0 \end{pmatrix}, \, \begin{pmatrix}0 & 0 & 1 \\ 1 & 0 & 0 \\ 0 & 1 & 0 \end{pmatrix}, \, \begin{pmatrix}0 & 0 & 1 \\ 0 & 1 & 0 \\ 1 & 0 & 0 \end{pmatrix}, \, \begin{pmatrix}\frac{1}{2} & 0 & \frac{1}{2} \\ 0 & 1 & 0 \\ \frac{1}{2} & 0 & \frac{1}{2} \end{pmatrix}\]
The first $5$ are the $123$-avoiding permutations in $\mathfrak{S}_3$, but the last is not integral.
\end{example}

Thus, $\mathcal{B}_n^k$ is just a rational polytope in general. Our main result is the following:

\begin{thm} \label{thm:main}
For any $1\leq k \leq n$, the Ehrhart quasi-polynomial $L(\mathcal{B}_n^k;t)$ is actually a polynomial. In other words, $\mathcal{B}_n^k$ exhibits Ehrhart period collapse.
\end{thm}

\begin{example}
Continuing \cref{ex:b32}, the Ehrhart ``quasi''-polynomial of $\mathcal{B}_3^2$ is
\[L(\mathcal{B}_3^2;t) = \frac{1}{12}t^4 + \frac{1}{2}t^3 + \frac{17}{12}t^2 + 2t + 1 = 2\binom{t+3}{4} + \binom{t+2}{2}.\]
\end{example}

\begin{example}
The polytope $\mathcal{B}_5^3$ has $1232$ vertices, one of which is:
{\small \[\renewcommand*{\arraystretch}{1}\begin{pmatrix} 0 & 1/5 & 4/5 & 0 & 0 \\ 3/5 & 0 & 1/5 & 0 & 1/5 \\ 0 & 3/5 & 0 & 2/5 & 0 \\ 2/5 & 0 & 0 & 3/5 & 0 \\ 0 & 1/5 & 0 & 0 & 4/5 \end{pmatrix}\]}
Its Ehrhart ``quasi''-polynomial is
{\footnotesize \begin{gather*} 
 L(\mathcal{B}_5^3;t) = \frac{1553003}{6974263296000} t^{16} + \frac{661643}{74724249600} t^{15} + \frac{8186069}{49816166400} t^{14} + \frac{542779}{287400960} t^{13} +  \\  \frac{192164881}{12773376000} t^{12} + \frac{23022107}{261273600} t^{11} + \frac{958262023}{2438553600} t^{10} + \frac{7126361}{5225472} t^9 + \frac{26012124739}{6967296000} t^8 + \frac{4245350429}{522547200} t^7 \\ + \frac{8999625683}{638668800} t^6 + \frac{698212357}{35925120} t^5 + \frac{3070672490609}{145297152000} t^4 + \frac{9229105657}{518918400} t^3 + \frac{64192517}{5765760} t^2 + \frac{677}{144} t + 1.
\end{gather*}}
\end{example}

Even though \cref{thm:main} provides new examples of polytopes exhibiting Ehrhart period collapse, there is a sense in which these are examples are also not so new. This is because our proof of \cref{thm:main} is based on comparing the restricted Birkhoff polytope to a certain Gelfand--Tsetlin polytope. 

Specifically, we show that there is a continuous, piecewise-linear bijection from the restricted Birkhoff polytope $\mathcal{B}_n^k$ to a Gelfand--Tsetlin polytope $\mathcal{M}_n^k$ which moreover gives a bijection on lattice points (and, importantly, on lattice points of all dilates). Thus, $L(\mathcal{B}_n^k;t)$ and $L(\mathcal{M}_n^k;t)$ are the same. As we mentioned, the Ehrhart quasi-polynomial of any Gelfand--Tsetlin polytope is known to actually be a polynomial, so~$L(\mathcal{B}_n^k;t)$ is a polynomial as well. 

We note however that~$\mathcal{B}_n^k$ and~$\mathcal{M}_n^k$ are \emph{not} combinatorially equivalent (see~\cref{ex:comb_eq}), so in that sense the polytopes we produce are indeed new. Also, the bijection from $\mathcal{B}_n^k$ to $\mathcal{M}_n^k$ is essentially the \dfn{Robinson--Schensted--Knuth (RSK) correspondence}, a highly nontrivial construction from the theory of symmetric functions.

After reviewing some background on Gelfand--Tsetlin polytopes and the RSK correspondence in \cref{sec:gt,sec:rsk}, we proceed to prove \cref{thm:main} in the way we have indicated in \cref{sec:proof}. In \cref{sec:poset} we give a different perspective on the polytopes $\mathcal{B}_n^k$ and~$\mathcal{M}_n^k$ which relates them to poset polytopes. Finally, in \cref{sec:final} we conclude with some final remarks and questions.

\subsection*{Acknowledgments} This collaboration originated on MathOverflow~\cite{MO}. The use of Sage mathematical software~\cite{Sage} was also crucial for this work. Finally, we thank the anonymous referees for useful comments which improved the article.

\section{Gelfand--Tsetlin polytopes} \label{sec:gt}

A size $n$ \dfn{Gelfand--Tsetlin (GT) pattern} is a triangular array of numbers
\[ G = \begin{array}{c c c c c} g_{1,1} & g_{1,2} & g_{1,3} & \cdots & g_{1,n} \\ 
 & g_{2,2} & g_{2,3} & \cdots & g_{2,n} \\ & & g_{3,3} & \cdots & g_{3,n} \\ & & & \ddots & \vdots \\ & & & & g_{n,n} \end{array}\]
for which $g_{i,j} \geq g_{i,j+1} \geq g_{i+1,j+1}$ for all $i,j$. In other words, the entries must weakly decrease along rows and down columns. Here we allow the $g_{i,j}$ to be real numbers, but we say $G$ is integral if $g_{i,j}\in\mathbb{Z}$ for all $i,j$.

The \dfn{shape} of the GT pattern $G$ is~$(g_{1,1},g_{2,2},\ldots,g_{n,n})$. For $0 \leq \ell \leq n-1$, we let~$d_{\ell}(G)$ denote the sum of the entries of~$G$ along the diagonal where $j-i=\ell$, i.e.,
\[ d_{\ell}(G) \coloneqq \sum_{i=1}^{n-\ell}x_{i,\ell+i}.\]
The \dfn{content} of $G$ is $(d_{n-1}(G)-d_n(G),\, d_{n-2}(G)-d_{n-1}(G), \, \ldots,\, d_0(G)-d_1(G))$, with the convention $d_n(G)\coloneqq 0$.

Let $\lambda = (\lambda_1,\ldots,\lambda_n),\mu=(\mu_1,\ldots,\mu_n)\in \mathbb{N}^n$ be vectors of nonnegative integers. We will consider GT patterns of shape $\lambda$ and content $\mu$. If the entries of $\lambda$ are not weakly decreasing then there are no GT patterns of shape~$\lambda$, so from now on let us assume that $\lambda=(\lambda_1\geq \lambda_2 \geq \cdots \geq \lambda_n \geq 0)$ is an integer partition. For a similar reason, we also assume that $\mu_1+\cdots+\mu_n = \lambda_1 +\cdots+\lambda_n$.

It is well-known that the integral Gelfand--Tsetlin patterns of shape~$\lambda$ and content~$\mu$ are in bijection with semistandard Young tableaux of shape $\lambda$ and content~$\mu$. Hence, their number is given by the \dfn{Kostka number} $K_{\lambda,\mu}$. This is the dimension of the $\mu$ weight space in the irreducible $\mathrm{GL}_n$ representation $V^{\lambda}$.

The \dfn{Gelfand--Tsetlin polytope} $\mathcal{GT}_{\lambda,\mu}$ consists of all GT patterns of shape $\lambda$ and content $\mu$. Thus, the number of lattice points of $\mathcal{GT}_{\lambda,\mu}$ is $K_{\lambda,\mu}$. As we mentioned, $\mathcal{GT}_{\lambda,\mu}$ has non-integral vertices in general~\cite{deloera2004vertices, king2004stretched}. Nevertheless:

\begin{thm}[\cite{billey2004vector, kirillov2001ubiquity}] \label{thm:gt}
For any $\lambda, \mu$, $L(\mathcal{GT}_{\lambda,\mu}; t)$ is a polynomial.
\end{thm}

In the language of representation theory, \cref{thm:gt} says that the ``stretched'' Kostka numbers $K_{t\lambda,t\mu}$ are polynomial in $t$. It is known more generally that stretched Littlewood-Richardson coefficients $c_{t\lambda,t\mu}^{t\nu}$ are polynomial in $t$~\cite{derksen2002littlewood, rassart2004polynomiality}.

We explained in \cref{sec:intro} that our strategy for proving \cref{thm:main} is to show that~$\mathcal{B}_n^k$ has the same Ehrhart quasi-polynomial as a certain Gelfand--Tsetlin polytope $\mathcal{M}_n^k$, and then appeal to \cref{thm:gt}. However, it turns out to be convenient to present the partner polytope $\mathcal{M}_n^k$ in a slightly different way from how we just described GT patterns.

Specifically, we define $\mathcal{M}_n^k$ to be the set of $n\times n$ $\mathbb{R}$-matrices $Y=(y_{i,j})_{i,j=1}^{n}$ satisfying the inequalities
\begin{align*}
y_{i,j} \leq y_{i+1,j} &\textrm{ \; and \; } y_{i,j} \leq y_{i,j+1} \textrm{ \, for all $i,j$}, \\
y_{1,1} \geq 0 &\textrm{ \; and \; } y_{n,n} \leq k,
\end{align*}
as well as the equalities
\begin{align*}
(d_{n-1}(Y),d_{n-2}(Y),\ldots,d_0(Y)) &= (1,2,\ldots,n), \\
(d_0(Y),d_{-1}(Y),\ldots,d_{-n+1}(Y)) &= (n,n-1,\ldots,1).
\end{align*}
Here we again use $d_\ell(Y)$ to denote the sum of entries of $Y$ along the diagonal~$j-i=\ell$: $d_{\ell}(Y) \coloneqq \sum_{i=1}^{n-\ell}y_{i,\ell+i}$ for $0 \leq \ell \leq n-1$, while $d_{\ell}(Y) \coloneqq \sum_{i=1}^{n+\ell}y_{i-\ell,i}$ for~$-n+1 \leq \ell \leq 0$.

Note that, unlike GT patterns, the matrices in $\mathcal{M}_n^k$ have entries which are weakly \emph{increasing} along rows and down columns, but this is just a matter of convenience. We will now show that $\mathcal{M}_n^k$ is really a Gelfand--Tsetlin polytope in disguise. 

Two polytopes $\mathcal{P}$ in $\mathbb{R}^n$ and $\mathcal{Q}$ in $\mathbb{R}^m$ are said to be \dfn{integrally equivalent} if there is an affine transformation $\varphi\colon \mathbb{R}^n \to \mathbb{R}^m$ which restricts to a bijection $\varphi\colon \mathcal{P}\to \mathcal{Q}$ and also restricts to a bijection $\varphi\colon \mathrm{aff}(\mathcal{P})\cap\mathbb{Z}^n \to \mathrm{aff}(\mathcal{Q})\cap\mathbb{Z}^m$, where $\mathrm{aff}$ denotes affine span. Integrally equivalent polytopes are combinatorially equivalent and have the same Ehrhart functions (see, e.g., \cite[\S4]{stanley2002polytope}).

\begin{lemma} \label{lem:gt_equiv}
For any $1\leq k \leq n$, the polytope $\mathcal{M}_n^k$ is integrally equivalent to $\mathcal{GT}_{\lambda,\mu}$ where $\lambda=(k^n,0^n)$ and $\mu=(1^n,(k-1)^n)$. (Here we use multiplicity notation, i.e., $k^n$ means the part~$k$ repeated $n$ times.)
\end{lemma}

\begin{proof}
Let $\lambda$, $\mu$ be as in the statement of the lemma and consider what a Gelfand--Tsetlin pattern $G$ of shape $\lambda$ and content $\mu$ must look like. The main diagonal being $\lambda=(k^n,0^n)$ together with the requirement that entries decrease along rows and down columns forces a triangle of $k$'s and a triangle of $0$'s. For concreteness let us depict the case $n=3$:
\[ G = \begin{array}{c c c c c c} k & k & k & * & * & * \\ 
 & k & k & * & * & * \\ & & k & * & * & * \\ & & & 0 & 0 & 0 \\ & & & & 0 & 0 \\ & & & & & 0 \end{array}\]
Now consider what the other entries (the $n\times n$ pattern of $*$'s) can be. They are weakly decreasing along rows and down columns. They are at most $k$ and at least~$0$. And, because of the condition that the GT pattern has content $\mu=(1^n,(k-1)^n)$, their diagonal sums form the sequence $1,2,\ldots,n$ and then $n-1,n-2,\ldots,1$. There are no other requirements, so such a pattern of~$*$'s is exactly an $180^\circ$ rotation of a matrix in $\mathcal{M}_n^k$. Thus, $\mathcal{M}_n^k$ is integrally equivalent to~$\mathcal{GT}_{\lambda,\mu}$, as claimed.
\end{proof}

In the next two sections we will define a continuous, piecewise-linear bijection from $\mathcal{B}_n^k$ to $\mathcal{M}_n^k$ which will prove that $L(\mathcal{B}_n^k;t)=L(\mathcal{M}_n^k;t)$. But let us take a moment to observe that $\mathcal{B}_n^k$ and $\mathcal{M}_n^k$ are \emph{not} combinatorially equivalent.

\begin{example} \label{ex:comb_eq}
The polytope $\mathcal{M}_3^2$ has $7$ vertices:
\begin{gather*}
\begin{pmatrix}1 & 1 & 1 \\ 1 & 1 & 1 \\ 1 & 1 & 1 \end{pmatrix}, \, \begin{pmatrix}0 & 1 & 1 \\ 1 & 1 & 1 \\ 1 & 1 & 2 \end{pmatrix}, \, \begin{pmatrix}0 & 1 & 1 \\ 0 & 1 & 1 \\ 1 & 2 & 2 \end{pmatrix}, \, \begin{pmatrix}0 & 0 & 1 \\ 1 & 1 & 2 \\ 1 & 1 & 2 \end{pmatrix}, \, \begin{pmatrix}0 & 0 & 1 \\ 0 & 1 & 2 \\ 1 & 2 & 2 \end{pmatrix}, \\ 
\renewcommand*{\arraystretch}{1.5} \begin{pmatrix}\frac{1}{2} & \frac{1}{2} & 1 \\ \frac{1}{2} & \frac{1}{2} & \frac{3}{2} \\ 1 & \frac{3}{2} & 2 \end{pmatrix}, \, \begin{pmatrix}0 & \frac{1}{2} & 1 \\ \frac{1}{2} & \frac{3}{2} & \frac{3}{2} \\ 1 & \frac{3}{2} & \frac{3}{2} \end{pmatrix}
\end{gather*}
Recall from \cref{ex:b32} that $\mathcal{B}_3^2$ has $6$ vertices. So indeed $\mathcal{M}_3^2$ and $\mathcal{B}_3^2$ are not combinatorially equivalent. However, we can also check that they both have $5$ lattice points, in agreement with $L(\mathcal{B}_3^2;t)=L(\mathcal{M}_3^2;t)$.
\end{example}

We conclude this section with a few more basic comments comparing $\mathcal{B}_n^k$ and~$\mathcal{M}_n^k$. While $\mathcal{B}_n^n = \mathcal{B}_n$ is always an integral polytope, its partner $\mathcal{M}_n^n$ is not integral already for $n = 3$. For instance, the (non-integral) vertex of $\mathcal{M}_3^2$ listed last in \cref{ex:comb_eq} is also a vertex of $\mathcal{M}_3^3$. On the other hand, both $\mathcal{B}_n^1$ and $\mathcal{M}_n^1$ are single lattice points.

Next, consider the case $k=2$. We explained above that the lattice points of $\mathcal{B}_n^2$ are the $123$-avoiding permutations in $\mathfrak{S}_n$, of which there are Catalan number many. Via~\cref{lem:gt_equiv} we see that lattice points of $\mathcal{M}_n^2$ are in bijection with $n\times 2$ standard Young tableaux, another set classically enumerated by the Catalan numbers.

Finally, we note that because $\mathcal{B}_n^k$ is a subset of a $0,1$-polytope, it is clear that all its lattice points are actually vertices. Using an argument similar to~\cite[Proposition~9]{alexandersson2016gelfand}, it can also be shown that all the lattice points of~$\mathcal{M}_n^k$ are vertices. Alternatively, we can argue that the lattice points of~$\mathcal{M}_n^k$ are vertices as follows. Consider the following unimodular linear transformation $\varphi\colon \mathbb{R}^{n\times n} \to \mathbb{R}^{n\times n}$ (for concreteness, we depict the case~$n=4$):
\[ \varphi\colon \begin{pmatrix} x_{11} & x_{12} & x_{13} & x_{14} \\ x_{21} & x_{22} & x_{23} & x_{24} \\ x_{31} & x_{32} & x_{33} & x_{34} \\ x_{41} & x_{42} & x_{43} & x_{44} \end{pmatrix} \mapsto \begin{pmatrix} x_{11} & x_{12} & x_{13} & x_{14} \\ x_{21} & (x_{22}-x_{21}) & (x_{23}-x_{13}) & (x_{24}-x_{14}) \\ x_{31} & (x_{32}-x_{31}) & (x_{33}-x_{32}) & (x_{34}-x_{24}) \\ x_{41} & (x_{42}-x_{41}) & (x_{43}-x_{42}) & (x_{44}-x_{43}) \end{pmatrix} \]
Then the elements of $\varphi(\mathcal{M}_n^k)$ are matrices with nonnegative entries and diagonal sums all equal to $1$. Hence, $\varphi(\mathcal{M}_n^k)$ is a subset of a $0,1$-polytope, and so its lattice points are vertices. Because $\varphi$ is unimodular, this implies the same for $\mathcal{M}_n^k$.

\section{The RSK correspondence as a piecewise-linear map} \label{sec:rsk}

As we explained in \cref{sec:intro}, the bijection from $\mathcal{B}_n^k$ to $\mathcal{M}_n^k$ we require is a version of the RSK correspondence. We now review what we will need to know about RSK.

The \dfn{Robinson--Schensted--Knuth (RSK) correspondence} gives a bijection between $n\times n$ $\mathbb{N}$-matrices $X$ and pairs $(P,Q)$ of semistandard Young tableaux of the same shape with entries in $\{1,2,\ldots,n\}$. We refer the reader to~\cite[Chapter 7]{stanley1999ec2} for a textbook treatment  of the RSK correspondence. By viewing the output tableaux~$P$ and $Q$ as Gelfand--Tsetlin patterns, and by gluing these GT patterns along their main diagonal, we can produce a $n\times n$ $\mathbb{N}$-matrix~$Y$ which is weakly increasing along rows and down columns. In this way, we can view RSK as a construction whose inputs and outputs are both $n\times n$ $\mathbb{N}$-matrices. 

The crucial observation is that, viewed as a map between matrices in the way we have just indicated, RSK admits a \emph{piecewise-linear} description. This piecewise-linear description of RSK is due to Pak~\cite{pak2001hook}\footnote{See also the independent and contemporary work of Kirillov and Berenstein~\cite{kirillov1995, kirillov2001rsk}.}, and in fact it makes sense not just for $\mathbb{N}$-matrices but for $\mathbb{R}$-matrices. Without going through all the combinatorial details of its construction, we will summarize the essential properties of piecewise-linear RSK, which we denote $\rho$, in the following theorem.

Let $\mathcal{X}_n$ be the set of all $n\times n$ $\mathbb{R}$-matrices with nonnegative entries, i.e.,
\[ \mathcal{X}_n \coloneqq \{X = (x_{i,j})_{i,j=1}^{n}\colon x_{i,j} \geq 0 \textrm{ for all $i,j$}\}.\] 
Let $\mathcal{Y}_n$ be the set of all $n\times n$ $\mathbb{R}$-matrices with nonnegative entries which are weakly increasing in rows and columns, i.e.,
\[ \mathcal{Y}_n \coloneqq \{Y = (y_{i,j})_{i,j=1}^{n}\colon y_{1,1} \geq 0, \; y_{i,j} \leq y_{i+1,j} \textrm{ and } y_{i,j} \leq y_{i,j+1} \textrm{ for all $i,j$}\}. \]
Note that both $\mathcal{X}_n$ and $\mathcal{Y}_n$ are unbounded polyhedral cones in $\mathbb{R}^{n\times n}$.

\begin{thm}[Pak~\cite{pak2001hook}] \label{thm:rsk}
For any $n\geq 1$, there is a continuous, piecewise-linear bijection $\rho\colon \mathcal{X}_n\to \mathcal{Y}_n$ satisfying the following properties:
\begin{enumerate}
    \item \label{cond:homo} It is homogeneous in the sense that $\rho(tX)=t\rho(X)$ for all $X \in \mathcal{X}_n, t \geq 0$.
    
    \item \label{cond:lattice} It restricts to a bijection $\rho\colon \mathcal{X}_n \cap \mathbb{Z}^{n\times n} \to \mathcal{Y}_n \cap \mathbb{Z}^{n\times n}$ on lattice points.
    
    \item \label{cond:diag_sums} Let $X=(x_{i,j}) \in \mathcal{X}_n$ and $Y=(y_{i,j})\in \mathcal{Y}_n$ with $Y=\rho(X)$. 
    
    Let $r_i=r_i(X)$ and $c_j=c_j(X)$ be $X$'s row and column sums, respectively: $r_i(X) \coloneqq \sum_{j=1}^{n} x_{i,j}$ for $1 \leq i \leq n$ and $c_j(X) \coloneqq \sum_{i=1}^{n} x_{i,j}$ for $1 \leq j \leq n$. 
    
    Let $d_{\ell}=d_{\ell}(Y)$ be the diagonal sums of $Y$: $d_{\ell}(Y) \coloneqq \sum_{i=1}^{n-\ell}y_{i,\ell+i}$ for $0 \leq \ell \leq n-1$ and $d_{\ell}(Y) \coloneqq \sum_{i=1}^{n+\ell}y_{i-\ell,i}$ for $-n+1 \leq \ell \leq 0$.
    
    Then
    \begin{align*}
        (d_{n-1},d_{n-2},\ldots,d_0) &= (r_1, r_1 + r_2, \ldots, r_1+r_2 +\cdots + r_n), \\
        (d_{-n+1},d_{-n+2},\ldots,d_0) &= (c_1, c_1 + c_2, \ldots, c_1+c_2 +\cdots + c_n).
    \end{align*}
   
    \item \label{cond:lis} With $X$ and $Y$ as above,
    \[y_{n,n} = \max \; x_{i_1,j_1} + x_{i_2,j_2} + \cdots + x_{i_{2n-1},j_{2n-1}}, \]
    a maximum over all $(1,1)=(i_1,j_2) \lessdot (i_2,j_2) \lessdot \cdots \lessdot (i_{2n-1},j_{2n-1})=(n,n)$.
\end{enumerate}
\end{thm}

\begin{remark}
For those who know about the RSK correspondence but not its piecewise-linear description, let us briefly explain what these various properties mean in combinatorial terms. Property~\eqref{cond:lattice} simply says that we get a bijection between $\mathbb{N}$-matrices $X$ and pairs of tableaux $(P,Q)$ of the same shape. Property~\eqref{cond:diag_sums} says that the column sums of~$X$ give the content of $P$, and the row sums of~$X$ give the content of $Q$. Property~\eqref{cond:lis} is Schensted's theorem~\cite{schensted1961longest}, which says that the first part of the shape of $P$ records the length of the ``longest increasing subsequence'' of the input matrix~$X$. (The combinatorial meaning of Property~\eqref{cond:homo} is that scaling the input matrix $X$ corresponds to ``stretching'' the output tableaux $P$ and $Q$, which is true but is a less well-known property of RSK.)
\end{remark}

\begin{example} \label{ex:rsk}
For $n=2$ it is easy to see that $\rho$ must be
\[\rho\begin{pmatrix} x_{1,1} & x_{1,2} \\ x_{2,1} & x_{2,2} \end{pmatrix} = \begin{pmatrix} \min(x_{1,2},x_{2,1}) & x_{1,1}+x_{1,2} \\ x_{1,1}+x_{2,1} & x_{1,1}+x_{2,2}+\max(x_{1,2},x_{2,1})\end{pmatrix} \]
\end{example}

For $n \geq 3$, the formulas for the entries of $\rho(X)$ in terms of the entries of $X$ get very messy. The easiest approach to constructing $\rho$ is to define it inductively. In order to do this, one needs to consider as intermediary objects not just $n\times n$ $\mathbb{R}$-matrices, but fillings of Young diagrams of arbitrary shape with real numbers. Then, $\rho$ can be built up one box at a time. This inductive construction is explained clearly in~\cite{pak2001hook}; see also~\cite{hopkins2014rsk} for another expository account. At any rate, we just need the statement of \cref{thm:rsk} and not any specific formulas for $\rho$.

As mentioned, \cref{thm:rsk} is contained within Pak's paper~\cite{pak2001hook}, although he does not state it as a single proposition like we have. So let us take a moment to comment on where the various parts of \cref{thm:rsk} are discussed in~\cite{pak2001hook}. First note that Pak works ``slice-by-slice.'' That is, he fixes $\delta=(\delta_{-n+1},\ldots,\delta_0,\ldots,\delta_{n-1}) \in \mathbb{R}^{2n-1}$ and considers
\begin{align*} 
\mathcal{X}_{\delta} &\coloneqq \left\{X\in \mathcal{X}_n\colon \parbox{4.19in}{\small $(r_1(X),r_2(X),\ldots,r_n(X)) = (\delta_{n-1}-\delta_{n},\delta_{n-2}-\delta_{n-1},\ldots,\delta_0-\delta_1),$ \\ $(c_1(X),c_2(X),\ldots,c_n(X)) = (\delta_{-n+1}-\delta_{-n}, \delta_{-n+2}-\delta_{-n+1},\ldots,\delta_{0}-\delta_{-1})$} \right\}, \\
\mathcal{Y}_{\delta} &\coloneqq \{Y \in \mathcal{Y}_n\colon (d_{-n+1}(Y),\ldots,d_{n-1}(Y)) = \delta\}.
\end{align*}
(We use the convention $\delta_{-n}\coloneqq\delta_n\coloneqq 0$.) These are bounded polytopes, and Pak defines a bijection between these two polytopes for each fixed $\delta$. But the bijections for different~$\delta$ can be packaged together in the way that we have stated. Note as well that Pak defines the bijection in the other direction (i.e., from $\mathcal{Y}_{\delta}$ to $\mathcal{X}_{\delta}$). But again that is not an issue because he also explains what the inverse map looks like. 

These polytopes are defined in \cite[Section~3]{pak2001hook} and the inductive construction of the bijection $\rho$ between them is given in \cite[Section~4]{pak2001hook}.  Property~\eqref{cond:homo} of this bijection is clear from construction, Properties~\eqref{cond:lattice} and~\eqref{cond:diag_sums} are verified in \cite[Section~4]{pak2001hook}, and Property~\eqref{cond:lis} is \cite[Theorem 6]{pak2001hook}. Again, see~\cite{hopkins2014rsk} for another account of the construction of $\rho$.

\section{Proof of the main result} \label{sec:proof}

An easy corollary of \cref{thm:rsk} is:

\begin{cor} \label{cor:main}
For any $1\leq k \leq n$, there is a continuous, piecewise-linear bijection $\rho\colon \mathcal{B}_n^k \to \mathcal{M}_n^k$ that restricts to a bijection $\rho\colon \mathcal{B}_n^k \cap \frac{1}{t}\mathbb{Z}^{n\times n} \to \mathcal{M}_n^k \cap \frac{1}{t}\mathbb{Z}^{n\times n}$ for any integer $t\geq 1$.
\end{cor}
\begin{proof}
The $\rho$ in question is of course the restriction of the $\rho$ from \cref{thm:rsk}.

To see that this works, let $X \in \mathcal{B}_n^k\subseteq \mathcal{X}_n$. First note, thanks to Property~\eqref{cond:diag_sums}, that the row and column sums of $X$ all being equal to one is equivalent to the sequence of diagonal sums of $\rho(X)$ being $1,2,\ldots,n$ and then $n-1,n-2,\ldots,1$. Further, thanks to Property~\eqref{cond:lis}, the maximum sum of any chain of entries in $X$ being bounded by~$k$ is equivalent to the overall maximum entry of~$\rho(X)$ being at most~$k$. These conditions on $\rho(X)$ are precisely the equalities and inequalities which define the polytope $\mathcal{M}_n^k$ inside of~$\mathcal{Y}_n$, so indeed $\rho$ restricts to a bijection $\rho\colon \mathcal{B}_n^k \to \mathcal{M}_n^k$. 

The last part, about giving a bijection when we further intersect with $\frac{1}{t}\mathbb{Z}^{n\times n}$, follows from Properties~\eqref{cond:homo} and~\eqref{cond:lattice}.
\end{proof}

We can now prove our main result.

\begin{proof}[Proof of \cref{thm:main}]
\Cref{cor:main} says that $L(\mathcal{B}_n^k;t) = L(\mathcal{M}_n^k;t)$. By \cref{lem:gt_equiv} and \cref{thm:gt}, $L(\mathcal{M}_n^k;t)$ is a polynomial. So $L(\mathcal{B}_n^k;t)$ is also a polynomial.
\end{proof}

\begin{remark} \label{rem:contingency}
There is a straightforward ``contingency table'' / ``transportation polytope'' extension of our main result. Let $\alpha = (\alpha_1,\ldots,\alpha_n), \; \beta=(\beta_1,\ldots,\beta_n)\in\mathbb{N}^n$ be fixed nonnegative integer vectors with $\alpha_1+\cdots+\alpha_n = \beta_1+\cdots+\beta_n$. Define
\[ \mathcal{B}_{\alpha,\beta}\coloneqq \{X\in \mathcal{X}_n\colon (r_1(X),\ldots,r_n(X)) = \alpha, \; (c_1(X),\ldots,c_n(X))=\beta\}.\]
 Of course, $\mathcal{B}_{(1,\ldots,1),(1,\ldots,1)}=\mathcal{B}_n$. The lattice points of $\mathcal{B}_{\alpha,\beta}$ are called \dfn{contingency tables} with row sums $\alpha$ and column sums $\beta$, and $\mathcal{B}_{\alpha,\beta}$ itself is  called a \dfn{transportation polytope}~\cite{klee1968facets, deloera2014transportation}. We define the partner polytope to $\mathcal{B}_{\alpha,\beta}$ in an analogous way:
\[ \mathcal{M}_{\alpha,\beta}\coloneqq \left\{Y\in \mathcal{Y}_n\colon \parbox{3.75in}{\begin{center}$(d_{n-1}(Y),\ldots,d_0(Y)) = (\alpha_1,\alpha_1+\alpha_2,\ldots,\alpha_1+\cdots+\alpha_n),$ \\ $(d_{-n+1}(Y),\ldots,d_0(Y))=(\beta_1,\beta_1+\beta_2,\ldots,\beta_1+\cdots+\beta_n)$\end{center}}\right\}.\]
Let $r\coloneqq \max(\alpha_1,\ldots,\alpha_n,\beta_1,\ldots,\beta_n)$ and $s\coloneqq \alpha_1+\cdots+\alpha_n=\beta_1+\cdots+\beta_n$, and for $k=r,r+1,\ldots,s$, define the restricted polytopes:
\begin{align*}
    \mathcal{B}^k_{\alpha,\beta} &\coloneqq\left\{X\in \mathcal{B}_{\alpha,\beta}\colon \parbox{3.5in}{\begin{center}$x_{i_1,j_1}+\cdots+x_{i_{2n-1},j_{2n-1}} \leq k$ \; for all \\ $(1,1)=(i_1,j_1)\lessdot (i_2,j_2) \lessdot \cdots \lessdot (i_{2n-1},j_{2n-1})=(n,n)$\end{center}} \right\}, \\
    \mathcal{M}^k_{\alpha,\beta} &\coloneqq\{Y\in \mathcal{M}_{\alpha,\beta}\colon y_{n,n} \leq k\}.
\end{align*}
Note that $\mathcal{B}^s_{\alpha,\beta} =\mathcal{B}_{\alpha,\beta}$ and $\mathcal{M}^s_{\alpha,\beta} =\mathcal{M}_{\alpha,\beta}$. The analog of \cref{cor:main} holds for $\mathcal{B}^k_{\alpha,\beta}$ and $\mathcal{M}^k_{\alpha,\beta}$, so that $L(\mathcal{B}^k_{\alpha,\beta};t)=L(\mathcal{M}^k_{\alpha,\beta};t)$. Furthermore, the same argument as in \cref{lem:gt_equiv} shows that $\mathcal{M}^k_{\alpha,\beta}$ is integrally equivalent to $\mathcal{GT}_{\lambda,\mu}$ with $\lambda=(k^n,0^n)$ and $\mu=(\alpha_1,\, \alpha_2,\, \ldots,\, \alpha_n,\, (k-\beta_n),\, (k-\beta_{n-1}),\, \ldots, \, (k-\beta_1)).$ Hence, by \cref{thm:gt}, $L(\mathcal{M}^k_{\alpha,\beta};t)$ is a polynomial, and so $L(\mathcal{B}^k_{\alpha,\beta};t)$ is a polynomial as well.
\end{remark}

\begin{remark}
In order to prove \cref{thm:main} all we really needed to do was find a bijection between $t\mathcal{B}_n^k \cap \mathbb{Z}^{n\times n}$ and $t\mathcal{M}_n^k \cap \mathbb{Z}^{n\times n}$ for all integers $t \geq 1$. For this purpose, the usual RSK bijection (i.e., between $\mathbb{N}$-matrices and pairs of tableaux) suffices. We preferred to use the piecewise-linear RSK map $\rho$ because it gives a bijection between the polytopes themselves. This piecewise-linear bijection between $\mathcal{B}_n^k$ and $\mathcal{M}_n^k$ has a similar flavor to other well-known polytopal constructions, like Stanley's transfer map between order and chain polytopes (see the next section).
\end{remark}

\section{A poset polytope perspective} \label{sec:poset}

In this section we explain a different way of thinking about the polytopes $\mathcal{B}_n^k$ and~$\mathcal{M}_n^k$. Rather than view $\mathcal{B}_n^k$ as obtained from the Birkhoff polytope by imposing additional inequalities, we can view it as obtained from a chain polytope by imposing certain equalities. Similarly, rather than view $\mathcal{M}_n^k$ as a Gelfand--Tsetlin polytope, we can view it as obtained from an order polytope by imposing certain equalities. 

The chain and order polytopes are the two poset polytopes introduced by Stanley in~\cite{stanley1986two}. We now briefly review their construction and the basic facts about them. So in this section we let $P$ be a finite poset. We refer the reader to~\cite[Chapter 3]{stanley2012ec1} for all poset terminology and notation we do not define.

Recall that a \dfn{chain} in $P$ is a totally ordered sequence $p_1 < p_2 <\cdots < p_k \in P$. We use $\mathbb{R}^P$ to denote the set of functions $f\colon P \to \mathbb{R}$. The \dfn{chain polytope} $\mathcal{C}(P)$ is the set of $f \in \mathbb{R}^P$ satisfying
\begin{align*}
    f(p) \geq 0 &\textrm{\, for all $p\in P$}, \\
    f(p_1) + \cdots + f(p_k) \leq 1 &\textrm{\, for all chains $p_1 < \cdots < p_k \in P$}.
\end{align*}
Stanley proved that the vertices of $\mathcal{C}(P)$ are the indicator functions of antichains of~$P$. Recall that an \dfn{antichain} is a subset $A\subseteq P$ of pairwise incomparable elements.

The \dfn{order polytope} $\mathcal{O}(P)$ is the set of $f\in \mathbb{R}^P$ satisfying
\begin{align*}
    0\leq f(p) \leq 1 &\textrm{\, for all $p\in P$}, \\
    f(p) \leq f(q) &\textrm{\, for all $p\leq q \in P$}.
\end{align*}
Stanley proved that the vertices of $\mathcal{O}(P)$ are the indicator functions of order filters of $P$. Recall that an \dfn{order filter} is a subset $F\subseteq P$ that is upwards-closed, i.e., for which $p \in F$ and $q \geq p \in P$ implies $q \in F$.

The chain and order polytopes of $P$ are not combinatorially equivalent in general. Nevertheless, Stanley proved:

\begin{thm}[{Stanley~\cite[Theorem 3.2]{stanley1986two}}] \label{thm:transfer}
There is a continuous, piecewise-linear bijection $\phi\colon \mathcal{O}(P)\to\mathcal{C}(P)$ that restricts to a bijection $\phi\colon \mathcal{O}(P) \cap \frac{1}{t}\mathbb{Z}^{P} \to \mathcal{C}(P) \cap \frac{1}{t}\mathbb{Z}^{P}$ for any integer $t\geq 1$.
\end{thm}

The $\phi\colon \mathcal{O}(P)\to\mathcal{C}(P)$ from \cref{thm:transfer} is called the \dfn{transfer map}, and it is actually quite easy to define:
\[ \phi(f)(p) \coloneqq \begin{cases} f(p) &\textrm{if $p$ is minimal}, \\ f(p) - \max\{f(q)\colon q<p \in P\} &\textrm{otherwise}.\end{cases}\]
Observe how, at the level of vertices, $\phi$ is the canonical bijection between order filters and antichains which sends an order filter to its subset of minimal elements. The inverse transfer map $\phi^{-1}\colon \mathcal{C}(P)\to\mathcal{O}(P)$ also has a simple description:
\[ \phi^{-1}(f)(p) = \max\{f(p_1)+\cdots+f(p_k)\colon \textrm{ chains \, } p_1 < \cdots < p_k=p \in P\}. \]
At the level of vertices, $\phi^{-1}$ sends an antichain to its upwards closure.

In order to make the connection with the restricted Birkhoff polytope, we need to focus on one particular poset. We use $[n] \coloneqq \{1,2,\ldots,n\}$ to denote the $n$-element chain poset, and hence use $[n]\times [m]$ to denote the (Cartesian) product of two chains. The connection between poset polytopes and the restricted Birkhoff polytopes is then given by the following proposition, whose proof is straightforward.

\begin{prop}
View the elements of $\mathbb{R}^{[n]\times [n]}$ as $n\times n$ $\mathbb{R}$-matrices in the obvious way. Then,
\begin{align*}
\mathcal{B}_n^k &= \left\{X \in k\mathcal{C}([n]\times[n])\colon\parbox{2.1in}{\begin{center}$(r_1(X),\ldots,r_n(X)) = (1,\ldots,1),$ \\ $(c_1(X),\ldots,c_n(X))=(1,\ldots,1)$\end{center}}\right\}, \\
\mathcal{M}_n^k &= \left\{Y\in k\mathcal{O}([n]\times[n])\colon \parbox{3.25in}{\begin{center}$(d_{n-1}(Y),d_{n-2}(Y),\ldots,d_{0}(Y))=(1,2,\ldots,n),$ \\
$(d_{0}(Y),d_{-1}(Y),\ldots,d_{-n+1}(Y))=(n,n-1,\ldots,1)$\end{center}}\right\}.
\end{align*}
Here, as before, we use $r_i(X)$ (respectively, $c_j(X)$) to denote the row sums (resp., column sums) of a matrix $X$, and $d_\ell(Y)$ to denote the diagonal sums of a matrix~$Y$. 
\end{prop}

We also note that the $\mathcal{B}_{\alpha,\beta}^k$ and $\mathcal{M}_{\alpha,\beta}^k$ from \cref{rem:contingency} can be obtained similarly.

However, the transfer map does not play nicely with these intersections with affine subspaces. In particular, $\phi^{-1}$ fails to give a bijection between $\mathcal{B}_n^k$ and $\mathcal{M}_n^k$, as we show in the next example.

\begin{example}
Even for $n=2$, we have
\[\phi^{-1}\begin{pmatrix} x_{1,1} & x_{1,2} \\ x_{2,1} & x_{2,2} \end{pmatrix} = \begin{pmatrix} x_{1,1} & x_{1,1} + x_{1,2} \\ x_{1,1} + x_{2,1} & x_{1,1}+x_{2,2}+\max(x_{1,2},x_{2,1})\end{pmatrix},\]
which fails to map $\mathcal{B}_2^2$ onto $\mathcal{M}_2^2$. Compare to \cref{ex:rsk}.
\end{example}

There are other variants of the transfer map, such as for so-called ``marked'' order and chain polytopes~\cite{ardila2011gelfand}, but as far as we know none of these transfer map variants recovers the RSK bijection $\rho\colon \mathcal{B}_n^k \to \mathcal{M}_n^k$.

\section{Final remarks} \label{sec:final}

We conclude with some final thoughts and questions.

\begin{remark}
It is well-known that the Birkhoff polytope $\mathcal{B}_n$ lives in the $(n^2-2n+1)$-dimensional subspace of $\mathbb{R}^{n\times n}$ determined by the row and column sums equalling one, and within this subspace is a full-dimensional polytope with $n^2$-facets given by the nonnegative entries inequalities (again, see, e.g., \cite[Example~0.12]{ziegler1995lectures}). Recall that we defined the restricted Birkhoff polytope $\mathcal{B}_n^k$ from $\mathcal{B}_n$ by imposing $\binom{2n-2}{n-1}$ additional inequalities. It is not hard to see that for $k > 1$,  $\mathcal{B}_n^k$ remains $(n^2-2n+1)$-dimensional. But we do not know which of the additional inequalities are irredundant, i.e., we do not have a description of the facets of the restricted Birkhoff polytope. Of course, we also do not have a description of the vertices of the restricted Birkhoff polytope. \Cref{tab:facet_vertex_nums} records the number of facets and vertices of $\mathcal{B}_n^k$ for small values of $n$ and~$k$. It would be worthwhile to give an exact facet (or vertex) description of $\mathcal{B}_n^k$, but the data suggest the descriptions might not be so simple. Even in the potentially simplest case of $\mathcal{B}_n^2$, these sequences do not appear in the On-Line Encyclopedia of Integer Sequences (OEIS)~\cite{oeis}.
\end{remark}

\begin{table}
    \centering
   \begin{tabular}{c|c|c|c|c|c}
    $n$/$k$ & 1 & 2 & 3 & 4 & 5 \\ \hline
    1 & 1 \\ \cline{1-3}
    2 & 1 & 2 \\ \cline{1-4}
    3 & 1 & 9 & 9 \\ \cline{1-5}
    4 & 1 & 32 & 24 & 16 \\ \cline{1-6}
    5 & 1 & 91 & 77 & 41 & 25 \\
    \end{tabular} \qquad \qquad \begin{tabular}{c|c|c|c|c|c}
    $n$/$k$ & 1 & 2 & 3 & 4 & 5 \\ \hline
    1 & 1 \\ \cline{1-3}
    2 & 1 & 2 \\ \cline{1-4}
    3 & 1 & 6 & 6 \\ \cline{1-5}
    4 & 1 & 49 & 34 & 24 \\ \cline{1-6}
    5 & 1 & 3692 & 1232 & 187 & 120 \\
    \end{tabular} 
    \medskip
    \caption{The number of facets (left) and vertices (right) of the restricted Birkhoff polytope $\mathcal{B}_n^k$ for $1 \leq k \leq n \leq 5$.}
    \label{tab:facet_vertex_nums}
\end{table}

\begin{remark}
We can measure how ``surprising'' an instance of Ehrhart period collapse is by comparing how small the period of the quasi-polynomial is to how big the denominators of the vertices of the polytope are. In~\cite{deloera2004vertices}, De Loera and McAllister showed that the denominators for Gelfand--Tsetlin polytopes can be arbitrarily big. The denominators for restricted Birkhoff polytopes also grow unboundedly. We will sketch an argument for this assertion in the next paragraph.

For $n \geq 2$, we can show that the point $X\in \mathcal{B}_{2n+1}^{2n}$ given by $x_{i,i}=\frac{2n}{2n+1}$ and~$x_{i,i+n}=\frac{1}{2n+1}$ (with indices taken modulo $2n+1$) is a vertex. Indeed, if we write~$X=\frac{1}{2} Y+\frac{1}{2} Z$ for two other points $Y,Z\in \mathcal{B}_{2n+1}^{2n}$ then we have that $Y$ and~$Z$ have the same support as $X$. Together with the row and column sum condition, this implies that $Y$ and~$Z$ are constant along the diagonal. However, since $\sum y_{i,i}$ and $\sum z_{i,i}$ cannot exceed $2n$ we have $y_{1,1}\le \frac{2n}{2n+1}$ and $z_{1,1}\le \frac{2n}{2n+1}$. This tells us that $Y=Z=X$ and therefore $X$ is a vertex.

In the opposite direction, we could also ask for an \emph{upper} bound on the denominator of vertices of $\mathcal{B}_n^k$. Limited computational evidence suggests that $n$ could be such an upper bound, but we do not know why this should be the case.
\end{remark}

\begin{remark}
Let $\mathcal{P}$ be a $d$-dimensional integral polytope with Ehrhart polynomial $L(P;t) = c_d t^d + c_{d-1} t^{d-1} + \cdots + c_1 t + c_0$. There are geometric interpretations of the coefficients $c_d$ and $c_{d-1}$ which imply that they are positive (and trivially $c_0=1$), but other coefficients can be negative. The question of which polytopes $\mathcal{P}$ have all Ehrhart polynomial coefficients nonnegative has recently attracted a lot of attention. Such polytopes are called \dfn{Ehrhart positive}. See for instance the survey by Liu~\cite{liu2019positivity}, or~\cite[\S4]{ferroni2023examples}. Stanley conjectured~\cite[Conjecture 5.1.2]{liu2019positivity} that the Birkhoff polytope is Ehrhart positive, but this remains open.

For a rational polytope $\mathcal{P}$ with an Ehrhart quasi-polynomial which happens to be an actual polynomial, we can ask the same question about whether the coefficients of this polynomial are nonnegative. King, Tollu, and Toumazet~\cite{king2004stretched} conjectured that every Gelfand--Tsetlin polytope $\mathcal{GT}_{\lambda,\mu}$ is Ehrhart positive. In fact, they made a stronger positivity conjecture for stretched Littlewood-Richardson coefficient polynomials. These conjectures are open.

Ehrhart positivity for Gelfand--Tstelin polytopes would certainly imply Ehrhart positivity for the restricted Birkhoff polytopes. However, we caution that some prominent conjectures about Ehrhart positivity have very recently been resolved in the negative~\cite{ferroni2022matroids}. Indeed, data from small examples can often be misleading~\cite{alexandersson2019polytopes}. So we have no opinion on Ehrhart positivity for $\mathcal{B}_n^k$.
\end{remark}

\begin{remark}
The \dfn{permutohedron} $\mathcal{P}_n$ is the convex hull in $\mathbb{R}^n$ of all permutations of the vector $(1,2,\ldots,n)$. Like the Birkhoff polytope, it is another very widely studied polytope (see, e.g., \cite[Example~0.10]{ziegler1995lectures} and  \cite[\S9.3]{beck2015computing}). There is a canonical projection $\pi\colon \mathcal{B}_n\to \mathcal{P}_n$ which sends the permutation matrix $X_w$ to the one-line notation of $w$. It could be interesting to look at the image $\pi(\mathcal{B}_n^k)\subseteq \mathcal{P}_n$ of the restricted Birkhoff polytope under this projection. For instance, limited computational evidence suggests that $\pi(\mathcal{B}_n^k)$ might also exhibit Ehrhart period collapse.
\end{remark}

\begin{remark}
\dfn{Combinatorial mutation}~\cite{akhtar2012minkowski} is a certain operation on polytopes whose definition is motivated by the study of mirror symmetry in toric geometry. It is also closely related to mutation in cluster algebras. Combinatorial mutation of polytopes preserves things like the Ehrhart polynomial, although it is not an integral equivalence. In particular, Higashitani~\cite{higashitani2020two} recently showed that the transfer map between the order and chain polytopes of a poset is a series of combinatorial mutations. In the first version of this paper available online, we asked whether the piecewise-linear RSK bijection $\rho$ between $\mathcal{B}_n^k$ and $\mathcal{M}_n^k$ could also be realized as a series of combinatorial mutations. This was subsequently answered in the affirmative in the very recent paper~\cite{clarke2022restricted}. In~\cite{clarke2022restricted}, the authors also study further interesting combinatorial properties of the restricted Birkhoff polytopes and their relatives.
\end{remark}

\begin{remark}
\dfn{Rowmotion}~\cite{striker2012promotion,striker2017dynamical} is a certain invertible operator acting on the order filters (or equivalently, antichains) of any finite poset $P$. Piecewise-linear rowmotion is a piecewise-linear extension of rowmotion to the order polytope $\mathcal{O}(P)$~\cite{einstein2021combiantorial} (or equivalently, chain polytope $\mathcal{C}(P)$~\cite{joseph2019antichain}). For a generic $P$, piecewise-linear rowmotion will have infinite order, but Grinberg and Roby~\cite{grinberg2015iterative} showed that it has order $n+m$ when $P=[n]\times[m]$ is the product of two chains. 

The Stanley--Thomas word of an antichain of $[n]\times[m]$, or more generally, any point in $\mathcal{C}([n]\times[m])$, is a vector associated to this point which rotates under rowmotion~\cite{joseph2021birational}. The Stanley--Thomas word of a point in $\mathcal{C}([n]\times[m])$ is related to its row and column sums in a simple way, and consequently it is easy to see that $\frac{1}{2}\mathcal{B}_n^2 \subseteq \mathcal{C}([n]\times[n])$ is exactly the set of points whose Stanley--Thomas word is~$(\frac{1}{2},\frac{1}{2},\ldots,\frac{1}{2})$. This implies that piecewise-linear rowmotion restricts to an order $2n$ action on $\mathcal{B}_n^2$. Moreover, for any integer $t \geq 1$, it restricts to an order $2n$ action on the finite set $\mathcal{B}_n^2\cap\frac{1}{t}\mathbb{Z}^{n \times n}$. This action of piecewise-linear rowmotion on the $k=2$ restricted Birkhoff polytope and its rational points might be worth studying further.
\end{remark}

\begin{remark}
A popular theme in algebraic combinatorics over the past 30 years has been to take some combinatorial construction, like RSK or rowmotion, and realize it as a continuous, piecewise-linear map. Often the goal is then to further de-tropicalize to obtain a birational map defined by a subtraction-free rational expression (see, e.g.,~\cite{kirillov2001introduction, einstein2021combiantorial}). But hopefully we showed in this paper how these piecewise-linear combinatorial constructions can also be interesting just from the point of view of polytope theory. In addition to the aforementioned paper~\cite{clarke2022restricted}, another very recent paper which also studies applications to Ehrhart theory of these kind of piecewise-linear versions of combinatorial constructions is~\cite{johnson2022piecewise}.
\end{remark}

\bibliography{main}{}
\bibliographystyle{abbrv}

\end{document}